\documentclass[reqno,11pt]{amsart}
\usepackage{amscd,amssymb}
\usepackage{amsmath}
\usepackage{hyperref}
\usepackage{amsrefs}
\usepackage{comment}
\usepackage{tikz}
\usetikzlibrary{matrix}
\usepackage[toc,page]{appendix}
\usepackage{marginnote}

\numberwithin{equation}{section}

\theoremstyle{plain}
\newtheorem{theorem}[subsection]{Theorem}
\newtheorem{lemma}[subsection]{Lemma}
\newtheorem{corollary}[subsection]{Corollary}
\newtheorem{proposition}[subsection]{Proposition}

\theoremstyle{definition}
\newtheorem{definition}[subsection]{Definition}
\newtheorem{example}[subsection]{Example}

\theoremstyle{remark}
\newtheorem{remark}[subsection]{Remark}

\usepackage{amsmath,amsfonts,latexsym}

\newcommand{\F}{\mathbb{F}}

\DeclareMathOperator{\Ker}{Ker}

\DeclareMathOperator{\im}{im}      

\DeclareMathOperator{\Tr}{Tr}

\DeclareMathOperator{\ind}{ind}

\newcommand{\forget}[1]{}

\newcommand{\innerprod}[1]{\langle #1 \rangle}

\allowdisplaybreaks[2]

\newcommand{\NeumannN}{\mathcal{N}}

\newcommand{\RR}{\mathbb{R}}
\newcommand{\CC}{\mathbb{C}}
\newcommand{\ZZ}{\mathbb{Z}}
\newcommand{\calD}{\mathcal{D}}

\newcommand{\Id}{\operatorname{Id}}
\newcommand{\NGam}{\mathcal{N}\Gamma}
\newcommand{\<}{\langle}
\renewcommand{\>}{\rangle}
\newcommand{\hatD}{\widehat{D}}
\newcommand{\esssup}{\operatornamewithlimits{ess\ sup}}
\newcommand{\RE}{\mathrm{Re}}
\newcommand{\Dom}{\operatorname{Dom}}

\newcommand{\calF}{\mathcal{F}}
\newcommand{\calT}{\mathcal{T}}

\usepackage{anysize}
\begin{document}
\title[Von Neumann algebra valued Schr\"odinger operators]{Spectral Theory of von Neumann algebra valued differential operators over non-compact manifolds}
\author{Maxim Braverman${}^\dag$}
\address{Department of Mathematics,
Northeastern University,
Boston, MA 02115,
USA}

\email{maximbraverman@neu.edu}
\urladdr{www.math.neu.edu/~braverman/}

\author{Simone Cecchini}
\address{Department of Mathematics,
Northeastern University,
Boston, MA 02115,
USA}

\email{cecchini.s@husky.neu.edu}

\thanks{${}^\dag$Supported in part by the NSF grant DMS-1005888.}


\begin{abstract}
We provide criteria for self-adjointness and $\tau$-Fredholmness of first and second order differential operators acting on sections of infinite dimensional bundles, whose fibers are modules of finite type over a von Neumann algebra $A$ endowed with a trace $\tau$. We extend the Callias-type index to operators acting on sections of such bundles and show that this index is stable under compact perturbations.  
\end{abstract}
\maketitle

\section{Introduction}\label{S:introduction}

In this paper we extend some results about the spectral properties of first and second order differential operators on complete Riemannian manifolds (see \cites{GromovLawson83,Shubin99,BrMiSh02,Br2,Chernoff73,Ro-Be70,Ol1,Shubin92}) to operators acting on sections of certain infinite dimensional bundles, called $A$-Hilbert bundles of finite type, cf. Section~\ref{SS:A-Hilbert bundles}. As a main example we consider lifts of operators acting on a finite dimensional vector bundle over a complete Riemannian manifold $M$ to a Galois cover $\widetilde{M}$ of $M$.

Let $A$ be a von Neumann algebra with a finite, faithful, normal trace $\tau$. Let  $E^+$ and $E^-$ be $A$-Hilbert bundles of finite type over a complete Riemannian manifold $M$, cf. Section~\ref{SS:A-Hilbert bundles}. In particular, this means that the fibers of $E^\pm$  are Hilbert spaces endowed with an action of $A$. We consider a first order differential operator $D\colon C^\infty_c(M,E^+)\to C^\infty_c(M,E^-)$ (here $C^\infty_c$ denotes the space of smooth compactly supported sections). We also fix an $A$-linear bundle map $V:E^+\to E^-$. 

We give a criterion  for self-adjointness of the generalized Schr\"odinger operator $H_V:= D^*D+V$. Then we show that if there exists a constant $C>0$ such that $V(x)>C$ for all $x$ outside of a compact subset of $M$, then the {\em von Neumann spectral counting function} $N_\tau(\lambda;H_V)$ of $H_V$ is finite for all $\lambda<C$. For operators acting on finite dimensional bundles this result is obtained in \cite{Anghel93}. 

One of the motivations for our study is an attempt to extend Atiayh's $L^2$-index theory for covering manifolds, \cites{Atiyah76,Schick05}, to coverings of non-compact manifolds. As a first application to index theory we develop in  Section~\ref{S:Callias} a Callias-type index theory, \cites{Anghel93Callias,BottSeeley78,Bunke95,Callias78}, for operators acting on $A$-Hilbert bundles over non-compact manifolds.  

More specifically, given a bundle map $F:E^+\to E^-$ we provide a criterion for $\tau$-Fredholmness of the operator $D_F:= D+F$. The $\tau$-Fredholmness implies that the kernel of $D_F$ and $D_F^\ast$ have finite $\tau$-dimension, where $D_F^\ast$ is the formal adjoint of $D_F$. Hence, we can define the $\tau$-index $\ind_\tau{}D_F$.  We prove that this index is stable under compact perturbations of $F$. 
 
 Another possible application of our analysis,  which will be discussed elsewhere, is to the equivariant index theory of operators acting on a non-compact manifold with a proper action of a Lie group $G$. A significant progress was achieved recently in developing the equivariant index theory of operators  acting on sections of a finite dimensional bundle over such manifolds,  \cites{MaZhang-noncompact,MathaiHochs,MathaiZhang10,Br-index,BrNonCompactGroup,BrCano14,Paradan03,MaZhang14}.  These results  find applications in the study of geometric quantization, representations of reductive groups and other areas.  This paper provides the analytic tools needed to extend these results to operators acting on $A$-Hilbert bundles.

The paper is organized as follows. In Section~\ref{S:main} we formulate our main results. In Section~\ref{S:A-Hilbert bundles} we recall the basic properties of von Neumann algebras and $A$-Hilbert modules. In Sections~\ref{S:prsaD} and \ref{S:prsa} we prove criteria of self-adjointness for first and second order operators respectively. In Section~\ref{S:Fredholm} we prove criteria for finiteness of the spectral counting function $N_\tau(\lambda;H_V)$ and for Fredholmness of the operator $D_F:=D+F$. In Section~\ref{S:Callias} we introduce the Callias-type $\tau$-index and prove its stability. 

\subsection*{Acknowledgment} We are very grateful to Ognjen Milatovic for very careful reading of the manuscript and providing a lot of useful remarks and corrections. 

\section{The main results}\label{S:main}

In this section we formulate  the main results of the paper.  

\subsection{An operator of order one}\label{SS:IDir}
Let $A$ be a von Neumann algebra with a finite, faithful, normal  trace $\tau:A\to \CC$. Let  $E=E^+\oplus{}E^-$ be a $\ZZ_2$-graded $A$-Hilbert bundle of finite type over a complete Riemannian manifold $M$, cf. Section~\ref{SS:A-Hilbert bundles}. In particular, this means that the fibers of $E^\pm$  are Hilbert spaces endowed with an action of $A$. We denote the Riemannian metric on $M$ by $g^{TM}$. 

Consider a first order differential operator $D\colon C^\infty_c(M,E^+)\to C^\infty_c(M,E^-)$ (here $C^\infty_c$ denotes the space of smooth compactly supported sections).  We assume that the principal symbol $\sigma(D)$ of $D$ is injective so that $D$ is elliptic. 

\subsection*{Assumption}
Throughout the paper we assume that there exists a constant $c>0$ such that
\begin{equation}\label{E:Istronglyelliptic}
	0\ <\ \|\sigma(D)(x,\xi)\| \ \le \ c\,|\xi|, \qquad \text{for all}\ \  x\in M, \ \xi\in T_x^*M\backslash\{0\}.
\end{equation}
Here $|\xi|$ denotes the length of $\xi$ defined by the Riemannian metric on $M$,  $\sigma(D)(x,\xi):E^+_x\to E^-_x$ is the leading symbol of $D$,  and $\|\sigma(D)(x,\xi)\|$ is its operator norm. 

An interesting class of examples of operators satisfying \eqref{E:Istronglyelliptic} is given by Dirac-type operators on $M$ paired with a connection on an $A$-Hilbert bundle, cf.  \cite{Schick05}*{Section 7.4}.

\subsection{Self-adjointness of a first order operator}\label{SS:IsaD}
Let $d\mu$ be a smooth measure on  $M$. We don't assume that $d\mu$ is the measure defined by the Riemannian metric $g^{TM}$. Let  $L^2(M,E^\pm)$ denote the space of square-integrable sections of $E^\pm$. We also set $E:=E^+\oplus{}E^-$. Then
\[
	 L^2(M,E)\ = \ L^2(M,E^+)\oplus L^2(M,E^-).
\]

Let $D^\ast$ denote the formal adjoint of $D$. Consider the operator
\begin{equation}\label{E:IcalD}
	\calD\ :=\ \begin{pmatrix}0&D^\ast\\D&0\end{pmatrix}.
\end{equation}
Then $\calD$ is an unbounded operator on $L^2(M,E)$. 

Our first result is the following extension of \cite[Th.~1.17]{GromovLawson83}:
\begin{theorem}\label{T:IsaD}
Suppose that  $D\colon C^\infty_c(M,E^+)\to C^\infty_c(M,E^-)$ satisfies \eqref{E:Istronglyelliptic} and the Riemannian metric $g^{TM}$ is complete. Then $\calD$ is essentially self-adjoint with initial domain $C^\infty_c(M,E)\ = \ C^\infty_c(M,E^+)\oplus C^\infty_c(M,E^-)$. 
\end{theorem}
The proof is given is Section~\ref{S:prsaD}.


\subsection{A Schr\"odinger-type operator}\label{SS:IShr}
Consider the Schr\"odinger-type operator 
\begin{equation}\label{E:IHV}
       H_V \ = \ D^*D \ + \ V,
\end{equation}
where $V(x):E_x^+\to E_x^+$ is an $A$-linear self-adjoint bundle endomorphism. 

We view $H_V$ as an unbounded operator on $L^2(M,E^+)$ with initial domain $C^\infty_c(M,E^+)$ and give a sufficient condition for self-adjointness of this operator. For simplicity we assume that the potential $V$ is a measurable section which belongs to $L^\infty_{loc}$.\footnote{Much more general potentials were considered in \cite{BrMiSh02}. It would be interesting to extend the results we present in this paper to the type of potentials considered in \cite{BrMiSh02}.} 

Our next result is the following extension of \cite[Th.~2.7]{BrMiSh02} to $A$-Hilbert bundles:

\begin{theorem}\label{T:Isa}
Suppose there exists a function $q\colon M\to\RR$ such that 

\begin{enumerate}
\item $V(x)\ge -q(x)\Id$ for all $x\in M$;
\item
$q\geq1$ and $q^{-1/2}$ is globally Lipschitz, i.e. there exists a constant $L>0$ such
that, for every $x_{1}, x_{2}\in M$,
\begin{equation}\label{E:ILipschitz}
       |q^{-1/2}(x_{1})-q^{-1/2}(x_{2})|\leq Ld(x_{1},x_{2}),
\end{equation}
where $d$ is the distance induced by the metric $g^{TM}$;

\item the metric $g:=q^{-1}g^{TM}$ on $M$ is complete.
\end{enumerate}
Then $H_V$ is essentially self-adjoint on $C^\infty_c(M,E^+)$.
\end{theorem}

The proof is given in Section~\ref{S:prsa}.

\begin{remark}
We say that a curve  $\gamma\colon [a,\infty)\to M$ {\em goes to infinity} if for any
compact set $K\subset M$ there exists $t_K>0$ such that $\gamma(t)\not\in K$, for all
$t\ge t_K$. Condition (3) of the theorem is equivalent to the condition that the integral $\int_\gamma\frac{ds}{\sqrt{q}}=\infty$ for every going to infinity curve $\gamma$.
\end{remark}

\begin{corollary}\label{C:Isaboudedbelow}
If $V(x)$ is bounded below, i.e. there exists a constant $b>0$ such that $V(x)\ge -b$, then the operator $H_V$ is essentially self-adjoint. 
\end{corollary}
In the rest of this paper we denote by $H_V$ both the operator and its self-adjoint closure.

\subsection{The spectral counting function of a Schr\"odinger-type operator}\label{SS:INlambda}
The  trace $\tau:A\to \CC$ on $A$ extends to a (possibly infinite) trace $\Tr_\tau$ on the space of bounded $A$-linear operators acting on $L^2(M,E^\pm)$. 

For each $\lambda\in \RR$ let $P_\lambda$ denote the orthogonal  projection onto the spectral subspace of the operator $H_V$ corresponding to the ray $(-\infty, \lambda]$ and define the {\em spectral counting function}
\begin{equation}\label{E:IN(lambda)}
	N_\tau(\lambda;H_V)\ := \ \Tr_\tau P_\lambda.
\end{equation}

\begin{theorem}\label{T:IN<infty}
Suppose there exist a compact $K\subset M$ and a constant $C$ such that 
\begin{equation}\label{E:IFredholmcondition}
	V(x)\ \ge \ C, \qquad \text{for all}\quad x\not\in K.
\end{equation}
Then 
\begin{equation}\label{E:IN<infty}
	N_\tau(\lambda;H_V)\ \ < \ \infty, \qquad \text{for all}\quad \lambda< C.
\end{equation}
\end{theorem}

\begin{remark}
Inequality \eqref{E:IFredholmcondition} implies that $V(x)$ is bounded below and, hence, the operator $H_V$ is self-adjoint by Corollary~\ref{C:Isaboudedbelow}.
\end{remark}


\subsection{$\tau$-Fredholmness}\label{SS:IFredholm}
We first recall the notion of Fredholmness in the von Neumann setting, cf. \cite[Section 3]{Breuer68Fredholm1}. 
Let $H_1$ and $H_2$ be $A$-Hilbert spaces and $T:\Dom(T)\subset H_1\rightarrow H_2$ be a closed $A$-operator.

\begin{definition}\label{D:Fredholm operator}
We say that the operator $T$ is \emph{$\tau$-Fredholm} if the following two conditions are satisfied:
		\begin{itemize}
			\item[(i)] $\Ker T$ has finite $\tau$-dimension;
			\item[(ii)] there exists an $A$-Hilbert subspace $L$ of $H_2$ such that $L\subset\im T$ and $\dim_\tau L^\perp<\infty$.
		\end{itemize}
\end{definition}

\begin{remark}
If $H_1$ and $H_2$ are Hilbert spaces over $\CC$, an operator $T:H_1\to H_2$ is called Fredholm if its kernel and cokernel are finite dimensional. Equivalently, this means that the  kernels of $T$ and $T^*$ are finite dimensional and the image of $T$ is closed. The image of a  $\tau$-Fredholm operator does not need to be closed. Moreover, the image of $T$ is not in general  an $A$-Hilbert space and the $\tau$-dimension of the cokernel of $T$ is not defined. Because of this we need to replace the condition of finite-dimensionality of the cokernel by Condition (ii) of Definition~\ref{D:Fredholm operator}. 
\end{remark}

Consider the operator 
\begin{equation}\label{E:Op1}	
	\mathcal{T}\ =\ \begin{pmatrix}0&T^*\\T&0\end{pmatrix}:\,H_1\oplus H_2\ \longrightarrow \ H_1\oplus H_2.
\end{equation}

\begin{lemma}\label{L:IFredholmness}
The operator $T$ is $\tau$-Fredholm if and only if there exists $\lambda>0$ such that 
\begin{equation}\label{E:NtaucalT}
	N_\tau(\lambda;\calT^2) \ < \ \infty.
\end{equation}
\end{lemma}
The lemma is proven in Section~\ref{SS:prIFredholmness}.

\begin{remark}\label{R:NtaucalT}
The inequality \eqref{E:NtaucalT} is equivalent to 
\[
	N_\tau(\lambda;T^*T)<\infty, \quad \text{and}\quad N_\tau(\lambda;TT^*)<\infty. 
\]
\end{remark}
\begin{remark} The condition \eqref{E:NtaucalT} is often taken as a definition of $\tau$-Fredholmness, cf. for example \cite{Luck02book}*{\S2.1.1}.
\end{remark}
\subsection{$\tau$-Fredholmness of a Callias-type operator}\label{SS:ICalliasFredholm}
Let $D$ be as in Section~\ref{SS:IDir} and let $F:E^+\to E^-$ be an $A$-linear bundle map. We set 
\[
	\calF\ :=\ \begin{pmatrix}
	0&F^*\\F&0
	\end{pmatrix}
\]
and consider the operator 
\begin{equation}\label{E:IDF}
	\calD_\calF\ := \ \calD+\calF\ = \ \begin{pmatrix}
	0&D^*+F^*\\D+F&0
	\end{pmatrix}.
\end{equation}
This operator satisfies the conditions of Theorem~\ref{T:IsaD} and, hence, is essentially self-adjoint.
Its self-adjoint closure is denoted by the same symbol $\calD_\calF$.

Let 
\[
	\{\calD,\calF\}\ := \ \calD\circ \calF\ +\ \calF\circ \calD \ = \ \begin{pmatrix}
	D^*F+F^*D&0\\0&DF^*+FD^*
	\end{pmatrix}
\]
denote the {\em anti-commutator} of $\calD$ and $\calF$. We also set $D_F:=D+F$.

\begin{definition}\label{D:Calliastype}
We say that the operators   $D_F$  and $\calD_\calF$ are of {\em Callias-type} if 
\begin{enumerate}
\item the leading symbol $\sigma(\calD)$ anti-commutes with $\calF$ so that the anti-commutator $\{\calD,\calF\}$ is an $A$-linear bundle endomorphism of $E=E^+\oplus{}E^-$; 
\item there exist a compact set $K\subset M$ and $\epsilon>0$ such that 
\begin{equation}\label{E:Calliastype}
	\calF^2(x)\geq\left\|\left\{\calD,\calF\right\}(x)\right\|+\epsilon
\end{equation}
for all $x\in M\backslash K$. 
\end{enumerate}
Here $\left\|\left\{\calD,\calF\right\}(x)\right\|$ denotes the norm of the A-linear map $\left\{\calD,\calF\right\}(x):E_x\to E_x$.
\end{definition}

In Section~\ref{S:Fredholm} we get the following corollary of Theorem~\ref{T:IN<infty}:
\begin{theorem}\label{T:IFredholm}
A Callias-type operator $D_F$ is $\tau$-Fredholm. 
\end{theorem}

\subsection{A Callias-type index}\label{SS:ICallias}
Suppose now that $D_F$ is an operator of Callias-type. By Theorem ~\ref{T:IFredholm},  $\dim_\tau\Ker D_F$ and $\dim_\tau\Ker D_F^*$ are finite. Hence, we can define the $\tau$-index of $D_F$ by 
\begin{equation}\label{E:ICallias index}
	\ind_\tau D_F\ := \ \dim_\tau \Ker D_F\ - \ \dim_\tau\Ker D_F^*.
\end{equation}
For operators acting on sections of  finite-dimensional bundles (i.e. when the von Neumann algebra $A=\CC$) this index was introduced by C.~ Callias \cite{Callias78} and was extended and studied in many papers including \cite{Bunke95,BottSeeley78,Anghel93Callias}. We refer to \eqref{E:ICallias index} as the {\em Callias-type $\tau$-index} of the pair $(D,F)$. 

In Section~\ref{S:Callias} we prove the following stability property of the Callias-type $\tau$-index:
\begin{theorem}\label{T:ICallias}
Let $F_0$ and $F_1$ be two $A$-linear bundle maps $E^+\to E^-$ satisfying conditions $(1)$ and $(2)$ of Definition~\ref{D:Calliastype}. 
If there exists a compact set $K\subset M$ such that $F_0(x)=F_1(x)$ for all $x\not\in K$, then 
\begin{equation}\label{E:ICallias stability}
	\ind_\tau D_{F_0}\ = \ \ind_\tau D_{F_1}.
\end{equation}
\end{theorem}

\section{$A$-Hilbert bundles}\label{S:A-Hilbert bundles}

In this section we recall some basic properties of von Neumann algebras, cf. \cite{Blackadar06book}, and recall the definitions of $A$-Hilbert space and $A$-Hilbert bundle, cf. \cite{Schick05}. 

\subsection{A Hilbert space completion}\label{SS:completion}
Let $A$ be a von Neumann algebra endowed with a finite, faithful, normal  trace $\tau:A\to \CC$, cf. \cite[III.2.5]{Blackadar06book}. We normalize the trace so that $\tau(\Id)=1$.

We define on $A$ the inner product 
\begin{equation}\label{von neumann inner product}
	\left\langle a,b\right\rangle_\tau\ :=\ \tau(ab^\ast), \qquad a,b\in A,
\end{equation}
and  denote by $l^2(A)$ the Hilbert space completion of $A$ with respect to this inner product.
Notice that $l^2(A)$ is an Hilbert space endowed with an $A$-module structure.

\begin{example}\label{Ex:Hilbert Gamma}
Suppose $\Gamma$ is a discrete group and denote by $l^2(\Gamma)$ the Hilbert space  of complex valued square summable functions on $\Gamma$. The {\em right regular representation} of $\Gamma$ is the unitary representation $R:\Gamma\to \mathcal{B}(l^2(\Gamma))$ defined by 
\[
	(R_gu)(h)=u(hg),\qquad h,\,g\in \Gamma;\ u\in l^2(\Gamma).
\]
The smallest subalgebra of $\mathcal{B}(l^2(\Gamma))$ which is weakly closed and contains all the operators $R_g\ (g\in \Gamma)$  is called the {\em von Neumann algebra} of $\Gamma$ and is denoted by $\NGam$.

On $\NeumannN\Gamma$ we have the canonical faithful positive trace $\tau$ defined by:
\begin{equation}\label{E:canonical trace}
			\tau(f)=\innerprod{f(\delta_e),\delta_e}_{l^2(\Gamma)},\ \ \ \ f\in \NGam,
\end{equation}
where $\delta_e\in l^2(\Gamma)$ is by definition the characteristic function of the unit element.

The completion $l^2(\NGam)$ of $\NGam$ with respect to this inner product is canonically isomorphic to $l^2(\Gamma)$, 
cf. \cite[Example 7.11]{Schick05}.
\end{example}

\subsection{Ан $A$-Hilbert spaces}\label{SS:A-Hilbert spaces}
Let $A$ be a von Neumnn algebra endowed with a trace $\tau$ satisfying the same properties as in section~\ref{SS:completion}.
Suppose $A$ acts on a Hilbert space $H$.
The action of $A$ on $H$ is said  to be \emph{compatible} with the inner product if 
\begin{equation}\label{E:compatibleA}
		\langle xa^*,y \rangle = \langle x,ya \rangle,\qquad a \in {A}; \ x,y \in H.
\end{equation}
Notice, that the action of $\NGam$ on $l^2(\Gamma)$ in Example~\ref{Ex:Hilbert Gamma} is compatible with the inner product on $l^2(\Gamma)$.

\begin{definition}\label{D:A-Hilbert space}
An \emph{$A$-Hilbert space} is a Hilbert space $H$ endowed with a compatible $A$-module structure such that there exists a separable Hilbert space $V_H$ and  an isometric $A$-linear embedding 
\[
	  i: H\ \hookrightarrow\ l^2(A)\otimes V_H.
\]
We say that $H$ is \emph{an $A$-Hilbert space of finite type} if $V_H$ can be chosen finite dimensional. 
\end{definition}
Notice, that since the action of $A$ is compatible with the scalar product, the orthogonal complement $i(H)^\perp$  to $i(H)$ in $ l^2(A)\otimes V_H$ is also an $A$-module. In other words $H$ is a projective $A$-module, cf. \cite{Blackadar06book}.


\subsection{A-Hilbert bundles}\label{SS:A-Hilbert bundles}

\begin{definition}
An \emph{$A$-Hilbert bundle} $E$ on a manifold $M$ is a locally trivial bundle of $A$-Hilbert spaces, the transition functions being $A$-Hilbert space isomorphisms. 
	
If the fibers are $A$-Hilbert spaces of finite type, the bundle is called an  \emph{$A$-Hilbert bundle of finite type}.

We denote by $C^\infty_c(M,E)$ the space of smooth compactly supported sections of $E$. 

If $M$ is endowed with a smooth measure $d\mu$ we define the $L^2$-scalar product 
\[
	(s_1,s_2)_2\ := \ \int_M\,\< s_1(x),s_2(x)\,\>\,d\mu, \qquad s_1,s_2\in C^\infty_c(M,E),
\]
and by $L^2(M,E)$ the completion of $C^\infty_c(M,E)$ with respect to this scalar product. 

We set $|s(x)|:=\<s(x),s(x)\>^{1/2}$ and denote by  
\[
	\|s\|\ :=\ \left(\,\int_M\,|s|^2\,d\mu\,\right)^{1/2}
\]
the $L^2$-norm of $s$. 
\end{definition}

\begin{example}\label{Galois covers}
Let $M$ be a smooth compact manifold and $\pi\colon \widetilde{M}\to M$ be the Galois cover of $M$ with deck transformation group $\Gamma$. We let $\Gamma$ act on the Hilbert space $l^2(\Gamma)$ by left and right convolution.
	
Let $A=\NGam$ denote  the von Neumann algebra of $\Gamma$. The right action of $\Gamma$ extends to a right action of $A$ on $l^2(\Gamma)$ commuting with the left $\Gamma$-action. Therefore, $E:=\widetilde M\times_\Gamma l^2(\Gamma)$ is an $A$-Hilbert bundle of finite type on $M$.

Notice that the space $L^2(M,E)$ coincides with the space $L^2(\widetilde{M})$  of square-integrable functions on $\widetilde{M}$.
\end{example}

\section{Self-adjointness of first order differential operators}\label{S:prsaD}

In this section we prove Theorem~\ref{T:IsaD}.

Since the operator $\calD$ is formally self-adjoint, to show that it is essentially self-adjoint we need  to prove that its maximal and minimal extensions coincide. Since the domain $\Dom(\calD_{\min})$ of the minimal extension is closed in the operator norm of $\calD$, it is enough to show that for every $s\in \Dom(\calD_{\max})$ there exists a sequence $\{s_k\}$ in $\Dom(\calD_{\min})$ such that 
\begin{equation}\label{E:sktos}
	\lim_{k\to\infty}s_k=s, \quad \lim_{k\to\infty}\calD s_k= \calD s, 
\end{equation}
where the limits are in $L^2$-norm topology.

\subsection{The minimal extension}\label{SS:minimal extension}
Recall that the Sobolev spaces of sections of an $A$-vector bundle were defined by  Mi{\v{s}}{\v{c}}enko and Fomenko \cite{FomenkoMoscenko80}. In particular, if $\Omega\subset M$ is an open set with compact closure, the Sobolev space $H^s_0(\Omega,E)$ is defined as the closure in Sobolev norm of the space $C_0^\infty(\Omega,E)$ of smooth sections, having compact support in $\Omega$. If $T$ is a differential operator of order $k$, then $T$ extends to a bounded operator
\[
	T:\,H^k_0(\Omega,E)\ \to \ L^2(M,E).
\]
Recall that the minimal domain $\Dom(T_{\min})$ is the closure of $C^\infty_c(M,E)$ with respect to the graph norm of $T$. We conclude that 
\begin{equation}\label{E:Hk subset Dommin}
	H^k_0(\Omega,E)\ \subset \ \Dom(T_{\min})
\end{equation}
for any open set $\Omega$ whose closure is compact. 

\subsection{The maximal extension}\label{SS:maximal extension}
Recall that the domain $\Dom(T_{\max})$ of the maximal extension consists of all sections $s\in L^2(M,E)$ such that $T{s}\in L^2(M,E)$, where $T{s}$ is understood in distributional sense.

Since $\calD$ is a first order elliptic operator, $\Dom(\calD_{\max})\subset H^1_{loc}(M,E)$. It follows now from \eqref{E:Hk subset Dommin}, that  for any Lipschitz function $\phi\in C^0_c(M)$ and any $s\in  \Dom(\calD_{\max})$ 
\begin{equation}\label{E:phis in Dommin}
		\phi\,s\ \in \ \Dom(\calD_{\min}).
\end{equation}

 \subsection{Proof of Theorem~\ref{T:IsaD}}\label{SS:prsaD}
By Lemma~8.9 of \cite{BrMiSh02} there exists a sequence $\{\phi_k\}$ of Lipschitz functions with compact support on $M$ such that 
\begin{equation}\label{E:cutoff}
	\begin{aligned}
	  (i&)\quad 0\leq\phi_k\leq1; \\
	  (ii&)\quad |d\phi_k|\leq\frac{1}{k};\\
       (iii&)\quad \lim_{k\to\infty}\phi_k(x)=1,\ \ \text{for al} \ \ x\in M.
   \end{aligned}
\end{equation}

For any $s\in \Dom(\calD_{\max})$ set $s_k=\phi_k s$. Then $s_k\in \Dom(\calD_{\min})$ by \eqref{E:phis in Dommin} and $\lim_{k\to \infty}s_k=s$. It remains to show that $\calD{}s_k$ converges to $\calD{}s$ in the $L^2$-norm.

We have:
\begin{equation}\label{E:calDsk}
			\calD s_k\ =\ \phi_k\calD s\ +\ [\calD,\phi_k]\,s.
\end{equation}
Notice that 
\[
	[\calD,\phi_k](x)\ =\ -i\sigma(\calD)\big(x,d\phi_k(x)\big)
\]
is a bundle map. From \eqref{E:cutoff} and \eqref{E:Istronglyelliptic} we conclude that 
\[
	\big\|\,[\calD,\phi_k]\,s\,\big\| \ = \ 
	\big\|\, \sigma(\calD)(x,d\phi_k)\,s\,\big\| \ \le \ \frac{c}k\cdot\|s\|.
\]
Hence, $\lim_{k\to\infty}[\calD,\phi_k]=0$. Since $\phi_k\calD s\to \calD s$ in $L^2$-norm we obtain
\[
	\lim_{k\to\infty}  \calD s_k \ = \ \calD s.
\]	
The essential self-adjointness of the operator $\calD$ is proved. \hfill $\square$

\section{Self-adjointness of a Schr\"odinger-type operator}\label{S:prsa}

In this section we prove Theorem~\ref{T:Isa}. We use the notation of Section~\ref{SS:IShr}.

We denote by $H_{V,0}$ the restriction of the operator \eqref{E:IHV} to the space $C^\infty_c(M,E^+)$ of smooth compactly supported sections of $E^+$. Let $H_{V,0}^*$ denote the operator adjoint to $H_{V,0}$ and let $\Dom(H_{V,0}^*)$ denote its domain. 

Since the operator $H_{V,0}$ is symmetric, to show that its closure is self-adjoint it is enough to prove that 
\begin{equation}\label{E:symmetry}
	( H_Vs_1,s_2) \ = \ (s_1,H_Vs_2), \qquad s_1,s_2\in \Dom(H_{V,0}^*).
\end{equation}

To prove \eqref{E:symmetry} we need some information about the behavior of sections from  $\Dom(H_{V,0}^*)$ at infinity. This information is provided by the following

\begin{proposition}\label{P:domain}
Suppose $q:M\rightarrow \RR$ is a function satisfying the assumptions of Theorem~\ref{T:Isa}.
If $s\in \Dom(H_{V,0}^*)$, then $q^{-1/2}Ds$ is square integrable and
\begin{equation}\label{E:ineqt}
    \|q^{-1/2}Ds\|^2 
       \ \leq \ 2\,\Big(\, (1+2L^2)\,\|s\|^2+\|s\|\,\|H_Vs\|\,\Big),
\end{equation}
where $L$ is the Lipschitz constant introduced in \eqref{E:ILipschitz}.
\end{proposition}

\begin{remark}\label{R:domain}
For the Schr\"odinger operator on scalar valued functions on
$\RR^n$ an analogous lemma was established in \cite{Ro-Be70}. The proof was
adapted in \cites{Ol1,Ol2} to the case of a Riemannian manifold and to differential forms in \cite{Br2}. The case of a general operator $D$ and a singular potential $V$ was considered in \cite{BrMiSh02}.  
\end{remark}

\subsection{Regularity of sections from $\Dom(H_{V,0}^*)$}\label{SS:regiularity}
The theory of elliptic (pseudo)-differential operators on $A$-Hilbert bundles of finite type was developed in \cite{BFKM} (see also \cite{FomenkoMoscenko80} for a similar theory for bundles of finitely generated $A$-modules). In particular, the Sobolev spaces of sections of such bundles are introduced in these papers and it is shown that any $s\in \Dom(H_{V,0}^*)$ belongs to the Sobolev space $H^2_{loc}$. Hence, 
\begin{equation}\label{E:DsVs}
	Ds\ \in \ L^2_{loc}(M,E^-), \ \ Vs\ \in \ L^2_{loc}(M,E^+),\qquad\text{for any}\quad s\in \Dom(H_{V,0}^*).
\end{equation}
The new information provided by Proposition~\ref{P:domain} is about the rate of decay of $Ds$ at infinity. 

\begin{remark}
The equation \eqref{E:DsVs} is the only place in the proof of Theorem~\ref{T:Isa} where we use the fact that the fibers of $E^\pm$ are modules of finite type. The rest of the proof of Theorem~\ref{T:Isa} follows the lines of \cite[\S9]{BrMiSh02} with almost no changes. It is even simpler, since in \cite{BrMiSh02} much more singular potentials are considered.
\end{remark}

Set $\hatD=-i\sigma(D)$. Then 
\[
	D(\phi s)\ = \ \hatD(d\phi)s+\phi Ds.
\]
Note that $\widehat{D^*}=-(\hatD)^*$.

\subsection{Proof of Proposition~\ref{P:domain}}\label{SS:prdomain}
Let $\psi$ be  a Lipschitz function with compact support such that a $0\leq\psi\leq q^{-1/2}\le 1$. 
Set 
\[
     C \ = \ \esssup_{x\in M}\, \|\hatD(d\psi)\|.
\]
Using \eqref{E:DsVs} we obtain
\begin{equation}\label{E:psiDs2}
	\begin{aligned}
     \|\psi Ds\|^2 \ & = \ (D^*(\psi^2Ds),s)
     \ = \
     (\psi^{2}D^*Ds,s) + 2(\psi\widehat{D^*}(d\psi)Ds,s)\\ 
		 &= \
     \RE(\psi^2D^*Ds,s) \ + \ 2\RE(\psi\widehat{D^*}(d\psi)Ds,s)\\
     & \le \ \RE(\psi^2D^*Ds,s) \ + \ 2C\|\psi Ds\|\|s\|\\
     & =\ \RE (\psi^2H_Vs,s)-(\psi^2Vs,s)+2C||\psi Ds\|\,\|s\|.
	\end{aligned}
\end{equation}
Since $V\ge -q\Id$, $q\ge 1$ and $q\psi^2\le 1$, we have
\begin{equation}\label{E:psiVss}
	(\psi^2Vs,s)\ = \ (V\psi s,\psi s) \ \ge\  - (q\psi s,\psi s) \ \ge \ -\| s\|^2.
\end{equation}
Using the inequality $ab\le \frac12 a^2+\frac12{b^2}$ we obtain
\begin{equation}\label{E:psiDss}
	2C\|\psi Ds\|\,\|s\| \ \le \ \frac12\,\|\psi Ds\|^2\ +\  2C^2\,\|s\|^2.
\end{equation}
Combining \eqref{E:psiDs2},\eqref{E:psiVss}, and \eqref{E:psiDss} and using that $\psi^2\le q^{-1}\le 1$ we get
\[
	\|\psi Ds\|^2\ \le \ \|H_Vs\|\,\|s\|\ +\ \frac12\,\|\psi Ds\|^2\ +\ (1+2C)\,\|s\|^2,
\]
and
\begin{equation}\label{E:psiDs<}
    \|\psi Ds\|^{2}
       \ \leq \ 2\,\Big(\, (1+2C^2)\,\|s\|^2+\|s\|\,\|H_Vs\|\,\Big).
\end{equation}

To prove \eqref{E:ineqt} we now make a special choice of the function $\psi$. Let $\phi_k$ be as in \eqref{E:cutoff} and set 
\(
	\psi_k :=  \phi_k\cdot q^{-1/2}.
\)
Then $0\leq\psi_{k}\leq q^{-1/2}$, and
\[
	|d\psi_{k}|\ \le\ |d\phi_k|\cdot q^{-1/2}\ +\ \phi_k|dq^{-1/2}|.
\] 
Therefore, $|d\psi_{k}|\leq\frac{1}{k}+L$.  Since
$\psi_{k}(x)\to q^{-1/2}(x)$ as $k\to\infty$, the dominated convergence theorem applied
to \eqref{E:psiDs<} with $\psi=\psi_k$ immediately implies~\eqref{E:ineqt}.
\hfill$\square$

\subsection{Proof of Theorem~\ref{T:Isa}}\label{SS:prsa}
Let $s_1,s_2\in \Dom(H_{V,0}^*)$. Then
\[
    (\phi s_1, D^*Ds_2) \ = \ (D(\phi s_1),Ds_2)
         \ = \
    (\hatD(d\phi)s_1,Ds_2)+(\phi Ds_1,Ds_2).
\]
Similarly, 
\[
     (D^*Ds_1,\phi s_2) \ = \  (Ds_1,\hatD(d\phi)s_2)+(\phi Ds_1,Ds_2)
\]
Hence, 
\begin{equation}\label{E:Huv-uHv}\notag
	(\phi s_1,H_Vs_2)-(H_Vs_1,\phi s_2)
     \ = \ (\hatD(d\phi)s_1,Ds_2)-(Ds_1,\hatD(d\phi)s_2).
\end{equation}

By \eqref{E:Istronglyelliptic},
\[
     \esssup_{x\in M}|\hatD(d\phi)|\ \leq\ c\; \esssup_{x\in M}|d\phi(x)|,
\]
Therefore,
\begin{equation}\label{E:pravlev}
      |(\phi s_1,H_Vs_2)-(H_Vs_1,\phi s_2)|
     \ \leq \ c\; \underset{x\in M}\esssup \left(|d\phi|q^{1/2}\right)\cdot
            \left(\|s_1\|\|q^{-1/2}Ds_2\|+\|s_2\|\|q^{-1/2}Ds_1\|\right).
\end{equation}

Consider a metric $g:=q^{-1}g^{TM}$. By condition (iii) of Theorem~\ref{T:Isa} this metric is complete.  By Lemma~8.9 of \cite{BrMiSh02} there exists a sequence $\{\phi_k\}$ of Lipschitz functions such that 
\begin{enumerate}
\item  $0\leq\phi_k\leq1$ and $|d\phi_k|_g\leq\frac{1}{k}$;
\item $\lim_{k\to\infty}\phi_k(x)=1$, for all $x\in M$.
\end{enumerate}
Since
$|d\phi_k|_{g}=q^{1/2}|d\phi_k|$, we conclude
\[
	\esssup_{x\in M}(|d\phi_k|q^{1/2}(x)) \ \leq\ \frac{1}{k}.
\] 
Using \eqref{E:pravlev}, we obtain
\[
     |(\phi_k u,H_Vv)-(H_Vu,\phi_k v)|
     \ \leq \ \frac{c}{k}\left(\|s_1\|\|q^{-1/2}Ds_2\|+\|q^{-1/2}Ds_1\|\|s_2\|\right)
        \ \to \  0, \quad\text{as} \quad k\to\infty,
\]
where the convergence to $0$ of the RHS of this inequality follows from Proposition~\ref{P:domain}.
On the other side, by the dominated convergence theorem we have
\[
     (\phi_k s_1,H_Vs_2)-(H_Vs_1,\phi_k s_2) \ \longrightarrow \ (s_1,H_Vs_2)-(H_Vs_1,s_2)
\]
as $k\to\infty$. Thus $(H_Vs_1,s_2)=(s_1,H_Vs_2)$, for all $s_1, s_2\in \Dom(H_{V,0}^*)$. Therefore, $H_V$ is essentially self-adjoint.\hfill$\square$

\section{Fredholmness} \label{S:Fredholm}

In this section we prove Theorems~\ref{T:IN<infty} and \ref{T:IFredholm}.

\subsection{Variational principle}\label{SS:variational}
We make use of a variational principle in the von Neumann setting, 
stated in the following:
\begin{lemma}\label{L:variational principle}
Let $H$ be an $A$-Hilbert space and $T$ be a self-adjoint operator commuting with the action of $A$.
Then, for every $\lambda\in\mathbb{R}$,
\begin{equation}\label{var.princ.eq.}
			N_\tau(\lambda;T)=\sup_L \dim_\tau L,
\end{equation}
where $L$ varies among the $A$-Hilbert subspaces of $H$ with $L\subset \Dom(T)$ and satisfying:
\[
	(Tu,u)\leq \lambda \,(u,u),\qquad  u\in L.
\]
\end{lemma}
This lemma is well-known.
For the proof we refer to \cite[Lemma 2.4]{Shubin96Morse} where the case $A=\NeumannN\Gamma$ is treated 
	(the case of a general finite von Neumann algebra $A$ follows with minor modifications).

\subsection{Restriction to a compact subset}\label{SS:restriction}
Note, first, that for any $A$-linear self-adjoint operator $T$ and any $l,\,\lambda\in \RR$ we have:
\begin{equation}
	N_\tau(\lambda+l;T+l)\ =\ N_\tau(\lambda;T).
\end{equation}
Therefore, by replacing $V(x)$ with $V(x)+l$ in Theorem~\ref{T:IN<infty},  we can assume that $C,\lambda>0$ in (\ref{E:IFredholmcondition}) and (\ref{E:IN<infty}) and  $V(x)\geq 0$ for all $x\in M$.

Let the compact set $K\subset M$ be  as in  \eqref{E:IFredholmcondition}. For any $\lambda>0$ consider the set
\[
	 M_\lambda \ := \ \{ x\in M:\, V(x)\geq \lambda\}.
\]
Then 
\[
	\Omega_\lambda\ :=\ M\backslash M_\lambda
\]
is an open set. We denote its closure by $\bar{\Omega}_\lambda$. We now assume that $0<\lambda<C$ and choose $\lambda_1$ such that $0<\lambda<\lambda_1<C$. Then 
\[
	\bar{\Omega}_{\lambda}\ \subset\  \Omega_{\lambda_1}, \qquad 
	\bar{\Omega}_{\lambda_1}\ \subset\  \Omega_{C}, \qquad \bar{\Omega}_C\subset K. 
\]

Let $\phi:M\to [0,1]$ be a smooth function such that 
\[
	\phi|_{\Omega_{\lambda_1}}\equiv 1, \qquad \phi|_{M\backslash\Omega_{C}}\equiv 0.
\]
Define the $A$-linear restriction map $\rho: L^2(M,E) \ \to \ L^2(\Omega_{C},E|_{\Omega_{C}})$ by the formula
\begin{equation}\label{E:MtoK1}
	\rho(s):= \phi s.
\end{equation} 

\begin{lemma}\label{L:restriction}
Let $L\subset \Dom(H_V)$ be an $A$-Hilbert subspace of $L^2(M,E)$ satisfying
\begin{equation}\label{E:Fredholmness:eq1}
			(H_Vs,s)\ \leq\ \lambda\,(s,s),\qquad s\in L.
\end{equation}
Then $\rho$ is injective when restricted to $L$ and $\rho(L)$ is a closed $A$-invariant subspace of $L^2(\Omega_\lambda,E|_{\Omega_\lambda})$. 
\end{lemma}
\begin{proof}
The potential $V$ satisfies the inequality $V(x)\geq 0$ for any $x\in M$. Hence, the operator $H_V$ satisfies the conditions of Theorem~\ref{T:Isa} with $q=\operatorname{const}$. It follows from Proposition~\ref{P:domain} that for any $s\in \Dom(H_V)$ we have $Ds\in L^2(M,E)$. Hence, 
\begin{equation}\label{E:HVss}
	(H_Vs,s)\ = \ \|Ds\|^2\ +\ (Vs,s)\ \ge\ (Vs,s).
\end{equation}
In particular, $(Vs,s)<\infty$. 

For any $s\in L$  using \eqref{E:HVss} we obtain
\[
	\lambda\,\|s\|^2\ \geq (H_Vs,s)\ \geq\   (Vs,s)\ \geq\ 
		\lambda_1\,\int_{M\setminus \Omega_{\lambda_1}} |s(x)|^2\, d\mu(x)
\]
and
\[ 
		\lambda\,\int_{\Omega_{\lambda_1}} |s(x)|^2\, d\mu(x)\ \ge\
		(\lambda_1-\lambda)\,\int_{M\setminus \Omega_{\lambda_1}} |s(x)|^2\, d\mu(x).
\]
Hence,
\[
	\|\rho(s)\|^2\ \geq \ \int_{\Omega_{\lambda_1}}\,|s(x)|^2\,d\mu(x) \ \geq\ 
	\frac{\lambda_1-\lambda}{\lambda}\, 
	             \int_{M\setminus \Omega_{\lambda_1}} |s(x)|^2\, d\mu(x).
\]
Therefore,
\begin{equation}\label{E:rho(s)>s}
	\begin{aligned}
		\|s\|^2\ & =  \int_{\Omega_{\lambda_1}} |s(x)|^2\, d\mu(x)+\,\int_{M\setminus \Omega_{\lambda_1}} |s(x)|^2\, d\mu(x)\\
		&\leq \|\rho(s)\|^2\ +\ \frac{\lambda}{\lambda_1-\lambda}\,\|\rho(s)\|^2
		= \left(\,\frac{\lambda_1}{\lambda_1-\lambda}\,\right)\, \|\rho(s)\|^2.
	\end{aligned}
\end{equation}
This inequality together with the fact that $\rho$ is a bounded $A$-equivariant map proves the lemma. 
\end{proof}

\subsection{Extension to a closed manifold}\label{SS:tilM}
Choose a closed manifold $\widehat{M}$  containing $\Omega_C$ as an open subset. Let $\widehat{E}=\widehat{E}^+\oplus\widehat{E}^-$ be a graded $A$-Hilbert bundle of finite type over $\widehat{M}$ extending $E|_{\Omega_C}$.  Let
$\widehat{D}:C^\infty(\widehat{E}^+)\rightarrow C^\infty(\widehat{E}^-)$ and $\widehat{V}:\widehat{E}^+\to \widehat{E}^+$  
 be a first order elliptic differential operator  and a positive bundle map which agree with $D$ and $V$ on $\Omega_C$. Set
 \[
 		H_{\widehat{V}} \ := \ \widehat{D}^*\widehat{D}\ + \ \widehat{V}.
 \]
Then the restrictions of $H_V$ and $H_{\widehat{V}}$ to $\Omega_C$ coincide. 

We view $\rho(L)\subset L^2(\Omega_C,E^+|_{\Omega_C})$ as a subspace of $L^2(\widehat{M},\widehat{E}^+)$.

\begin{lemma}\label{L:HVrhosrhos}
Under the assumptions of Lemma~\ref{L:restriction}, there exists a constant $R>0$ such that  for any $u\in \rho(L)$ we have 
\begin{equation}\label{E:HVrhosrhos}
	\big(\,H_{\widehat{V}}u,u\,\big) \ \leq\ R\,(u,u).
\end{equation}
\end{lemma}
\begin{proof}
Set 
\[
	a\ :=\ \max_{x\in M} \|\sigma(D)(x,d\phi(x))\|.
\]
Notice that the maximum exists, since $\phi$ has compact support.

For $u\in \rho(L)$ there exists $s\in L$ such that $u= \rho(s)= \phi{}s$. Then
\begin{equation}\label{E:tilDu<}
	\begin{aligned}
	\|\widehat{D} u\|^2\ &= \ \|D(\phi s)\|^2 \ = \ \big(\,\|\phi Ds\|+\|[D,\phi]s\|\,\big)^2\\ 
		&\leq \ \big(\,\|Ds\|+a\,\|s\|\,\big)^2 \ \leq \ 2\,\|Ds\|^2\ + \ 2a^2\,\|s\|^2.
	\end{aligned}
\end{equation}
Also, since $V(x)>0$ for all $x\in M$ we conclude
\begin{equation}\label{E:Vuu>}
	(\widehat{V} u,u)\ = \ (\phi^2Vs,s)\ \leq \ (Vs,s).
\end{equation}

By using \eqref{E:tilDu<} and \eqref{E:Vuu>} we obtain
\[
\begin{aligned}
	\big(\,H_{\widehat{V}}u,u\,\big) \ & =  \ 
	\|\widehat{D} u\|^2\ + \ (\widehat{V} u,u) \ \leq\ 
	2 \|Ds\|^2 \ + \ 2a^2\|s\|^2 \ + \ (Vs,s) \\ 
	&\leq   \ 
	2\,\big(\, \|Ds\|^2+ (Vs,s)\,\big) \  + \ 2a^2\|s\|^2 \ = \ 
	2\,(H_Vs,s) \ + \ 2a^2\,\|s\|^2.
\end{aligned}
\]
Using \eqref{E:Fredholmness:eq1} and \eqref{E:rho(s)>s} we now conclude 
\[
	\big(\,H_{\widehat{V}}u,u\,\big)\ \leq\ 2(\lambda+a^2)\,\|s\|^2 \ \leq\ 
	2(\lambda+a^2)\, \left(\,\frac{\lambda_1}{\lambda_1-\lambda}\,\right)\,\|u\|^2.
\]
Hence, \eqref{E:HVrhosrhos} holds with $R=	2(\lambda+a^2)\, \left(\,\frac{\lambda_1}{\lambda_1-\lambda}\,\right)$.
\end{proof}

\subsection{Proof of Theorem~\ref{T:IN<infty}}\label{SS:prIN<infty}
Since $\widehat{V}\ge0$, we have
\begin{equation}\label{E:spectral counting functions comparison}
	N_\tau(\lambda;H_{\widehat{V}})\leq N_\tau(\lambda;\widehat{D}^*\widehat{D}),\ \ \ \ \ \ \ \ \lambda\in\RR.
\end{equation}
It is shown in \cite{BFKM}*{\S2.4} that the spectral counting function $N_\tau(\,\lambda;\widehat{D}^*\widehat{D})$ is finite. Hence, it  follows from (\ref{E:spectral counting functions comparison}) and   Lemmas~\ref{L:variational principle} and \ref{L:HVrhosrhos} that 
\[
	\dim_\tau \rho(L)\ \le\ N_\tau(R;H_{\widehat{V}}) \ < \ \infty.
\]
From Lemma~\ref{L:restriction} and the open mapping theorem we deduce that the map $\rho|_L:L\rightarrow \rho(L)$ is an isomorphism of Hilbert $A$-spaces. By \cite[Theorem 1.12 (2)]{Luck02book} \begin{equation}\label{E:L=rho(L)}
	 \dim_\tau L\ =\ \dim_\tau \rho(L) \ \le \ N_\tau(R;H_{\widehat{V}}).
\end{equation} 
Hence, by Lemma~\ref{L:variational principle} we get
\[
	N_\tau(\lambda;H_V) \ \le\ N_\tau(R;H_{\widehat{V}}) \ < \ \infty.
\]
Theorem \ref{T:IN<infty} is proven.\hfill $\square$

\subsection{Proof of Lemma~\ref{L:IFredholmness}}\label{SS:prIFredholmness}
Consider the bounded operator
\begin{equation}
	\Phi(\mathcal{T}):=\mathcal{T}\,(I+\mathcal{T}^2)^{-\frac{1}{2}},
\end{equation}
where $\mathcal{T}$ is the operator defined in \eqref{E:Op1}.
Observe that the operator $\Phi(\mathcal{T})$ is $\tau$-Fredholm if and only if $T$ is $\tau$-Fredholm, and 
\[
	N_\tau(\lambda; \calT^2) \ = \ N_\tau\left(\,\frac{\lambda}{1+\lambda};\Phi(\calT)^2\,\right).
\]
Thus Lemma~\ref{L:IFredholmness} is a direct consequence of the following:
\begin{lemma}
	Let $S$ be a bounded self-adjoint $A$-linear operator on an $A$-Hilbert space $H$.
	Then $S$ is $\tau$-Fredholm if and only if there exists $\lambda>0$ such that
	\begin{equation}\label{E:NtaucalT2}
		N_\tau(\lambda;S^2) \ < \ \infty.
	\end{equation}
\end{lemma}
\begin{proof}
Suppose $N_\tau(\lambda;S^2) < \infty$ for some $\lambda>0$. We need to show that $S$ is $\tau$-Fredholm in the sense of  Definition~\ref{D:Fredholm operator}.

Since the function $N_\tau(\cdot;S^2)$ is nondecreasing, we have: 
\[
	\dim_\tau\Ker S\ =\ N_\tau(0;S^2) \ \leq \ N_\tau(\lambda;S^2)\ <\ \infty.
\]
Thus the condition (i) of Definition~\ref{D:Fredholm operator} is satisfied.

Set $L:=\im (I-P_\lambda(S^2))$, where $P_\lambda(S^2)$ denotes the orthogonal  projection onto the spectral subspace of the operator $S^2$ corresponding to $[0, \lambda]$.
Then $L$ is $A$-invariant, $L\subseteq \im S^2\subseteq \im S$ and
\[
	\dim_\tau L^\perp\ =\ \dim_\tau (P_\lambda(S))\ =\ N_\tau(\lambda;S^2)\ <\ \infty.
\]
Hence the condition (ii) of  Definition~\ref{D:Fredholm operator} is also satisfied.
	
\

Suppose now that $S$ is $\tau$-Fredholm and  let $L$ be an $A$-Hilbert subspace of $H$ such that $L\subseteq\im S$ and $\dim_\tau L^\perp<\infty$. 
	
The map 
\begin{equation}\label{E:S1}
	S_1\ :=\ S|_{(\Ker S)^\perp}:\,(\Ker S)^\perp\ \rightarrow\ H
\end{equation}
is one-to-one. 	
Set 
\[
	L_1\ := \ S_1^{-1}(L)\ \subset (\Ker S)^\perp.
\]
Then $L_1$ is a closed $A$-invariant subspace of $H$ and $S:L_1\to L$ is a bijection. It follows from the Open Mapping Theorem that there exists  $\epsilon>0$ such that 
\begin{equation}\label{E:Su>eps}
	\left\|\,Su\,\right\| \ > \ \epsilon\,\|u\|, \qquad \text{for any}\quad u\in L_1.
\end{equation}

We finish the proof of the lemma by showing that any $\lambda<\epsilon^2$  satisfies \eqref{E:NtaucalT2}. Recall that $P_\lambda(S^2)$ denotes the orthogonal projection on the spectral subspace of $S^2$ corresponding to the interval $[0,\lambda]$.  Thus   
\begin{equation}\label{E:Su<epsilon}
	\left\|\,Su\,\right\| \ \le \ \sqrt{\lambda}\,\|u\| \ < \ \epsilon\,\|u\|,
	\qquad \text{for  any}\quad u\in \im{}P_\lambda(S^2).
\end{equation}

From \eqref{E:Su>eps}  and \eqref{E:Su<epsilon} we now conclude that $\im P_\lambda(S^2)\cap L_1=\{0\}$.  Hence,  it follows from \cite[Theorem~1.12]{Luck02book} that 
\begin{equation}\label{E:Ntau<L1}
		N_\tau(\lambda;S^2)\ \leq\  \dim_\tau L_1^\perp.
\end{equation}
To finish the proof of the lemma it is now enough to show that $ \dim_\tau L_1^\perp<\infty$. 

Let $L_2$ be the orthogonal complement of $L_1$ in $(\Ker S)^\perp$. Then 
\begin{equation}\label{E:L1=L2+KerS}
	\dim_\tau L_1^\perp\ =\ \dim_\tau L_2\ +\ \dim_\tau\Ker S.
\end{equation}
Since $S$ is $\tau$-Fredholm, $\dim_\tau\Ker{}S$ is finite and it suffices to prove that $\dim_\tau L_2$ is finite. 

Notice that, by \eqref{E:Su>eps}, $S(L_2)$ is a closed subspace of $H$ and $S:L_2\to S(L_2)$ is a topological isomorphism. Hence, 
\begin{equation}\label{E:S(L2)=L2}
	\dim_\tau{}S(L_2)\ =\  \dim_\tau{}L_2,
\end{equation}
cf. \cite[Lemma~1.13]{Luck02book}. Since the map \eqref{E:S1} is one-to-one, we conclude that 
\[
	S(L_2)\cap L\ =\  S_1(L_2)\cap L \ = \ \{0\}.
\]
Let  $P_L$ be the orthogonal projection onto the subspace $L$. Then 
\[
	I-P_L:\,S(L_2)\ \to L^\perp
\]
is a one-to one map. Therefore, by \cite[Theorem~1.12(2)]{Luck02book}, 
\[
	\dim_\tau S(L_2)\ \leq\  \dim_\tau L^\perp\ <\ \infty.
\]
From \eqref{E:Ntau<L1}, \eqref{E:L1=L2+KerS}, and \eqref{E:S(L2)=L2} we now conclude that $N_\tau(\lambda;S^2)<\infty$. 
\end{proof}


\subsection{Proof of Theorem~\ref{T:IFredholm}}\label{SS:prIFredholm}
We are now ready to prove $\tau$-Fredholmness of a Callias-type operator $D+F$. The operator \eqref{E:IDF} satisfies the conditions of Theorem~\ref{T:IsaD} and, hence, is essentially self-adjoint. Hence
\begin{equation}\label{E:NDphi=NDphi2}
	N_\tau(\lambda;\calD_\calF) \ = \ N_\tau(\lambda^2;\calD_\calF^2).
\end{equation}
Moreover
\begin{equation}\label{E:calDPhi2}
	\calD_\calF^2\ = \ \calD^2 \ +\ \{\calD,\calF\}\ +\ \calF^2\ = \ \calD^2\ + \ V,
\end{equation}
where $V=\{\calD,\calF\}\ +\ \calF^2$. By the Callias condition \eqref{E:Calliastype}, $V(x)>\epsilon$ for all $x\in M\backslash{K}$. Hence, it follows from Theorem~\ref{T:IN<infty} that $N_\tau(\epsilon; \calD_\calF^2)$ is finite. Theorem~\ref{T:IFredholm} follows now from Lemma~\ref{L:IFredholmness}.

 \hfill$\square$

\section{Stability of the Callias-type $\tau$-index} \label{S:Callias}

In this section we prove Theorem~\ref{T:ICallias}.

\subsection{Continuous perturbations}\label{SS:continuous perturbation}
We start with an abstract result. Let $H_1$ and $H_2$ be $A$-Hilbert spaces and suppose that  $T:H_1\to H_2$ is a closed $A$-linear operator. We denote by $\Dom(T)$ the domain of $T$ considered as a Hilbert space with the graph scalar product
\[
	(x,y)_{\Dom(T)}\ :=\ (x,y)_{H_1}\ +\ (Dx,Dy)_{H_2}.
\]
Then $T:\Dom(T)\to H_2$ is a bounded $A$-linear operator. 

Assume in addition that $T$ is $\tau$-Fredholm, cf. Definition~\ref{D:Fredholm operator}.

\begin{definition}\label{D:boundedperturbation}
A one parameter family  $\{T_t\}_{t\in[0,1]}$ of $\tau$-Fredholm operators $T_t:H_1\to H_2$ is called a {\em continuous perturbation of $T$ with fixed domain} if $T_0=T$,  $\Dom(T_t)=\Dom(T)$ for all $t\in [0,1]$, and the induced family $T_t:\Dom(T)\to H_2$ is continuous in operator norm (as a family of maps between the Hilbert spaces $\Dom(T)$ and $H_2$). 
\end{definition}

The next lemma shows that the $\tau$-index is stable with respect to continuous perturbations with fixed domain.

\begin{lemma}\label{L:Abstract Stability}
Let $\{T_t\}_{t\in[0,1]}$ be a continuous perturbation of $T$ with fixed domain. Then 
\begin{equation}\label{E:Abstract Stability}
	\ind_\tau T_t\ = \ \ind_\tau T, \qquad\text{for all}\quad t\in [0,1].
\end{equation}
\end{lemma}

\begin{proof}
In general, note that the operator 
\[
	(I+T^*T)^{\frac{1}{2}}:\,\Dom(T)\ \longrightarrow\ H_1
\]
is an isometric isomorphism of $A$-Hilbert spaces and consider the family of operators 
\[
	\Phi(T_t):=T_t(I+T^*T)^{-\frac{1}{2}}:\,H_1\ \longrightarrow\ H_2.
\]
The operators $\Phi(T_t)$  form a continuous family of bounded $\tau$-Fredholm operators and 
\[
	\ind_\tau \Phi(T_t) \ = \ \ind_\tau T_t.
\]
Hence, it is enough to prove the lemma for the case when $\Dom(T)= H_1$  and $T_t$ is a  continuous family of bounded operators. In this case the lemma is proven in  \cite[Theorem 4]{Breuer69Fredholm2}.
\end{proof}

\subsection{Proof of Theorem~\ref{T:ICallias}}\label{SS:prICallias}
For $0\leq t\leq 1$, we set:
\[
	D_t\ :=\ D\ +\ F_0\ +\ t\,(F_1-F_0).
\]
Since the $A$-endomorphism $F_1-F_0$ vanishes outside of the compact set $K$, all $D_t$  satisfy the condition \eqref{E:Calliastype} and,  hence, are Callias-type operators. It follows from Theorem~\ref{T:IFredholm} that all the operators $D_t$ are $\tau$-Fredholm.

Since $F_1-F_0:L^2(M,E^+)\to L^2(M,E^-)$ is a bounded $A$-operator, the domain of $D_t$ is independent of $t$. Finally, we have:
\[
	\left\|D_s-D_t\right\|=\left|s-t\right|\left\|F_1-F_0\right\|,\qquad \text{for all}\quad s,\,t\in [0,1].
\]
Thus the family $\{D_t\}$ is continuous in operator norm.  Theorem~\ref{T:IFredholm} follows now from Lemma~\ref{L:Abstract Stability}.

\hfill$\square$


\end{document}